\documentclass{amsart}

\usepackage{stmaryrd}
\usepackage{graphics}
\usepackage{amssymb}

\newtheorem{lemma}{Lemma}
\newtheorem{theorem}{Theorem}
\newtheorem{corollary}{Corollary}
\newtheorem{proposition}{Proposition}
\newtheorem{question}{Question}
\theoremstyle{definition}
\newtheorem{definition}{Definition}
\newtheorem{problem}{Problem}
\newtheorem{example}{Example}
\newtheorem{assumption}{Running assumption}
\newcommand\csm[2]{\llbracket{#1},{#2}\rrbracket}
\newcommand\D[2]{\csm{#2}{\{1\}}\ominus\csm{#2}{{#1}\setminus{#2}}}
\newcommand\Fin{Fin}
\renewcommand\and{\text{ and }}
\begin{document}
\title{Compatibility support mappings in effect algebras}
\author{Gejza Jen\v ca}
\address{
Department of Mathematics and Descriptive Geometry\\
Faculty of Civil Engineering\\
Radlinsk\' eho 11\\
	Bratislava 813 68\\
	Slovak Republic
}
\email{gejza.jenca@stuba.sk}
\thanks{
This research is supported by grants VEGA G-1/0080/10 of M\v S SR,
Slovakia and by the Slovak Research and Development Agency under the contracts
No. APVT-51-032002, APVV-0071-06.
}
\begin{abstract}
We give a characterization of subsets of effect algebras, that can be
embedded into a range of an observable. To give this characterization,
we introduce a new notion of {\em compatibility support mappings.}
\end{abstract}
\maketitle

\section{Introduction and motivation}

\begin{question}
Let $S$ be a set of effects on a separable Hilbert space $\mathbb H$. Is there a
measurable space $(X,\mathcal A)$ and a POV-measure
$\alpha:(X,\mathcal A)\to\mathcal E(\mathbb H)$ such that
$S$ is a subset of the range of $\alpha$?
\end{question}

If $S$ consists only of orthogonal projections (that means, idempotent effects),
then the answer is simple: $S$ is
a subset of the range of a POV-measure iff the elements of $S$ commute.
On the other hand, if there are non-idempotent effects in $S$, the answer is not
known.

In the present paper, we examine a related question: 
\begin{question}
If $S$ is a subset of an effect algebra $E$, is there
a Boolean algebra $B$ and a morphism of effect algebras $\alpha:B\to E$
such that $S\subseteq\alpha(B)$?
\end{question}

This can be considered as a quantum-logical version of Question 1.
We prove
that, given subset $S$ of an effect algebra $E$ such that $1\in S$,
there exist a Boolean algebra $B$ and a morphism $\alpha:B\to E$ with
$S\subseteq\alpha(B)$ if and only if there is a mapping
$\csm{~.~}{~.~}:\Fin(S)\times\Fin(S)\to E$ satisfying certain properties.
We call them {\em compatibility support mappings}. The proof uses a 
modification of the limit techniques introduced in \cite{CatDAlGiuPul:EAaPM}.

We show that compatibility support mappings, and hence pairs $(B,\alpha)$,
exist whenever $S$ is an MV-algebra or $S$ is a pairwise 
commuting set of effects on a Hilbert space.
We prove several properties of strong 
compatibility support maps, 
generalizing the properties of the prototype Example \ref{ex:joinmeet}. 

The results presented in this paper are more general than the results
from an earlier paper \cite{Jen:CiIEA}, where only interval effect algebras
were considered. In that paper, a related notion of {\em witness 06} was introduced
to characterize coexistent subsets of interval effect algebras.

In the last section, we examine connections between 
compatibility support mappings and witness mappings.
We prove that, for a subset $S$ of an interval effect algebra,
every compatibility support map for $S$ gives rise to a witness mapping for $S$. 
We do not know whether this relationship is a one-to-one correspondence. 

\section{Definitions and basic relationships}

An {\em effect algebra} is a partial algebra $(E;\oplus,0,1)$ with a binary 
partial operation $\oplus$ and two nullary operations $0,1$ satisfying
the following conditions.
\begin{enumerate}
\item[(E1)]If $a\oplus b$ is defined, then $b\oplus a$ is defined and
		$a\oplus b=b\oplus a$.
\item[(E2)]If $a\oplus b$ and $(a\oplus b)\oplus c$ are defined, then
		$b\oplus c$ and $a\oplus(b\oplus c)$ are defined and
		$(a\oplus b)\oplus c=a\oplus(b\oplus c)$.
\item[(E3)]For every $a\in E$ there is a unique $a'\in E$ such that
		$a\oplus a'=1$.
\item[(E4)]If $a\oplus 1$ exists, then $a=0$
\end{enumerate}

Effect algebras were introduced by Foulis and Bennett in their paper 
\cite{FouBen:EAaUQL}. Independently, K\^ opka and Chovanec introduced
an essentially equivalent structure called {\em D-poset} (see \cite{KopCho:DP}).
Another equivalent structure, called {\em weak orthoalgebras} 
was introduced by Giuntini and Greuling in \cite{GiuGre:TaFLfUP}.

For brevity, we denote the effect algebra $(E,\oplus,0,1)$ by $E$.
In an effect algebra $E$, we write $a\leq b$ iff there is $c\in E$ such
that $a\oplus c=b$.
It is easy to check that every effect algebra is cancellative, thus
$\leq$ is a partial order on $E$. In this partial order,
$0$ is the least and $1$ is the greatest element of $E$.
Moreover, it is possible to introduce
a new partial operation $\ominus$; $b\ominus a$ is defined iff
$a\leq b$ and then $a\oplus(b\ominus a)=b$.
It can be proved that $a\oplus b$ is defined iff $a\leq b'$ iff
$b\leq a'$. It is usual to denote the domain of $\oplus$ by $\perp$.
If $a\perp b$, we say that $a$ and $b$ are {\em orthogonal}.

\begin{example}
The prototype example of an effect algebra is the {\em standard effect algebra
$\mathcal E(\mathbb H)$.} Let $\mathbb H$ be a Hilbert space. Let
$\mathcal S(\mathbb H)$ be the set of all bounded self-adjoint operators.
on $\mathbb H$. Let $\mathbb I$ be the identity operator $\mathbb H$.

For 
$A,B\in\mathcal S(\mathbb H)$, write $A\leq B$ if and only if,
for all $x\in\mathbb H$, $\langle Ax,x\rangle\leq\langle Bx,x\rangle$.

Put $\mathcal E(\mathbb H)=\{X\in\mathcal S(\mathbb H):0\leq X\leq I\}$
and for $A,B\in\mathcal E(\mathbb H)$ define $A\oplus B$ iff 
$A\oplus B\leq I$, $A\oplus B=A+B$. 
Then $(\mathcal E(\mathbb H),\oplus,0,I)$ is an effect algebra.
The elements of $\mathcal E(\mathbb H)$ are called Hilbert space effects.
\end{example}

An effect algebra $E$ is {\em lattice ordered} iff $(E,\leq)$ is a lattice.
An effect algebra is an {\em orthoalgebra} iff $a\perp a$ implies $a=0$.
An orthoalgebra that is lattice ordered is an orthomodular lattice.

An {\em MV-effect algebra} is a lattice ordered effect algebra $M$ in which,
for all $a,b\in M$, $(a\lor b)\ominus a=b\ominus (a\land b)$. It is proved
in \cite{ChoKop:BDP} that there is a natural, one-to one correspondence between
MV-effect algebras and MV-algebras given by the following rules.
Let $(M,\oplus,0,1)$ be an MV-effect algebra. Let $\boxplus$ 
be a total operation given by $x \boxplus y=x\oplus(x'\land y)$. Then
$(M,\boxplus,',0)$ is an MV-algebra. Similarly, let 
$(M,\boxplus,\lnot,0)$ be an MV-algebra. Restrict the operation
$\boxplus$ to the pairs $(x,y)$ satisfying $x\leq y'$ and call the
new partial operation $\oplus$. Then $(M,\oplus,0,\lnot 0)$ is an MV-effect algebra.

Among lattice ordered effect algebras, MV-effect algebras can be characterized
in a variety of ways. Three of them are given in the following
proposition.
\begin{proposition}
\cite{BenFou:PSEA}, \cite{ChoKop:BDP}
Let $E$ be a lattice ordered effect algebra. The following are equivalent
\begin{enumerate}
\item[(a)] $E$ is an MV-effect algebra.
\item[(b)] For all $a,b\in E$, $a\land b=0$ implies $a\leq b'$.
\item[(c)] For all $a,b\in E$, $a\ominus(a\land b)\leq b'$.
\item[(d)] For all $a,b\in E$, there exist $a_1,b_1,c\in E$ such that
$a_1\oplus b_1\oplus c$ exists, $a_1\oplus c=a$ and $b_1\oplus c=b$.
\end{enumerate}
\end{proposition}

Let $B$ be a Boolean algebra and let $E$ be an effect algebra.
An {\em observable} is a mapping $\alpha:B\to E$ such that
$\alpha(0)=0$, $\alpha(1)=1$ and for every $x,y\in B$ such that
$x\wedge y=0$, $\phi(x\vee y)=\phi(x)\oplus\phi(y)$.

\section{Compatibility support mappings --- definition and examples}

In this section we introduce (strong) 
compatibility support mappings and present
two examples.

\begin{definition}
\label{def:csm}
Let $E$ be an effect algebra, let $S\subseteq E$ be such that $1\in S$.
We say that $\csm{~.~}{~.~}:\Fin(S)\times\Fin(S)\to E$ is a {\em compatibility
support mapping for $S$} if and only if the following
conditions are satisfied.
\begin{enumerate}
\item[(a)]If $V_1\subseteq V_2$, then $\csm{U}{V_1}\leq\csm{U}{V_2}$.
\item[(b)]$\csm{U}{V}\leq\csm{U}{\{1\}}$.
\item[(c)]$\csm{U}{\emptyset}=0$.
\item[(d)]$\csm{\emptyset}{\{c\}}=c$.
\item[(e)]If $c\notin U\cup V$, then $\csm{U\cup\{c\}}{\{1\}}\ominus\csm{U\cup\{c\}}{V}=
	\csm{U}{V\cup\{c\}}\ominus\csm{U}{V}$
\end{enumerate}
A compatibility mapping is {\em strong} if and only if the following
condition is satisfied.
\begin{enumerate}
\item[(e*)]For all $c$, $\csm{U\cup\{c\}}{\{1\}}\ominus\csm{U\cup\{c\}}{V}=
	\csm{U}{V\cup\{c\}}\ominus\csm{U}{V}$
\end{enumerate}
Note that (e*) implies (e).
\end{definition}
\begin{example}
\label{ex:joinmeet}
Let $M$ be an MV-effect algebra. Define 
$\csm{~.~}{~.~}:\Fin(M)\times\Fin(M)\to E$
by
$$
\csm{U}{V}=(\bigwedge U)\wedge(\bigvee V).
$$
Then $\csm{~.~}{~.~}$ is a strong compatibility support mapping.
The conditions (a)-(d) are easy to prove.
Let us prove (e*).
\begin{align*}
\csm{U}{V\cup\{c\}}\ominus\csm{U}{V}=
(\bigwedge U)\wedge(c\vee(\bigvee V))\ominus
((\bigwedge U)\wedge(\bigvee V))=\\
=((\bigwedge U)\wedge c)\vee((\bigwedge U)\wedge(\bigvee V))\ominus
((\bigwedge U)\wedge(\bigvee V))=\\
=((\bigwedge U)\wedge c)\ominus((\bigwedge U)\wedge c\wedge(\bigvee V))=\\
=\csm{U\cup\{c\}}{\{1\}}\ominus\csm{U\cup\{c\}}{V}
\end{align*}
\end{example}
\begin{example}
\label{ex:product}
Let $\sqcup$ be an operation on the set of all
operators on a Hilbert space $\mathbb H$ given by
$$
a\sqcup b:=a+b-ab.
$$
It is easy to check that $\sqcup$ is associative with neutral element $0$.

If $a$ and $b$ are commuting effects, then $a.b$ is an effect with $a.b\leq a,b$. Moreover,
$a\sqcup b$ is an effect. Indeed, since $a,b$ are commuting effects, 
$1-a,1-b$ are commuting effects. Since $1-a,1-b$ are commuting effects,
$(1-a).(1-b)$ is an effect and
$$
1-(1-a).(1-b)=1-(1-a-b+ab)=a+b-ab
$$
is an effect.

Let $S$ be a set of commuting effects with $1\in S$;
there exists a commutative
$C^*$ algebra $A$ with $S\subseteq A$.
The operations $\sqcup, .$ are
commutative and associative
on $A\cap\mathcal E(\mathbb H)\supseteq S$.

Let $U,V$ be a finite subsets of $S$. Write $\bigsqcap U$ for the product of
elements of $U$. Write $\bigsqcup\emptyset=0$, $\bigsqcup\{c\}=c$ and,
for $V=\{v_1,\dots,v_n\}$ with  $n>1$, write 
$$
\bigsqcup V=v_1\sqcup\dots\sqcup v_n.
$$

Define
$\csm{~.~}{~.~}:\Fin(S)\times\Fin(S)\to E$
$$
\csm{U}{V}=(\bigsqcap U).(\bigsqcup V).
$$
Let us prove that
$\csm{~.~}{~.~}$ is a compatibility support mapping.

Proof of condition (a):
Suppose that $V_1\subseteq V_2$. We need to prove that $\csm{U}{V_1}\leq\csm{U}{V_2}$.
Let us prove that $\bigsqcup V_1\leq\bigsqcup V_2$.
Since $V_1\subseteq V_2$, we may write
\begin{align*}
\bigsqcup V_2=(\bigsqcup V_1)\sqcup (\bigsqcup (V_2\setminus V_1))=\\
=(\bigsqcup V_1)+(\bigsqcup (V_2\setminus V_1))
	-(\bigsqcup V_1).(\bigsqcup (V_2\setminus V_1))
\end{align*}
Therefore,
$$
(\bigsqcup V_2)-(\bigsqcup V_1)=
(\bigsqcup (V_2\setminus V_1))-(\bigsqcup V_1).(\bigsqcup (V_2\setminus V_1))\geq 0,
$$
so $\bigsqcup V_1\leq \bigsqcup V_2$. Since $\bigsqcup V_1\leq \bigsqcup V_2$,
$$
\csm{U}{V_1}=(\bigsqcap U).(\bigsqcup V_1)\leq(\bigsqcap U).(\bigsqcup V_1)=\csm{U}{V_2}.
$$

The conditions (b)-(d) are trivially satisfied.

Proof of condition the (e):
\begin{align*}
\csm{U}{V\cup\{c\}}-\csm{U}{V}=
(\bigsqcap U).(c\sqcup\bigsqcup V)-(\bigsqcap U).(\bigsqcup V)=\\
=(\bigsqcap U).(c+\bigsqcup V-c.(\bigsqcup V))-(\bigsqcap U).(\bigsqcup V)=\\
=(\bigsqcap U).c+(\bigsqcap U).(\bigsqcup V)-(\bigsqcap U).c.(\bigsqcup V)-
	(\bigsqcap U).(\bigsqcup V)=\\
=(\bigsqcap U).c-(\bigsqcap U).c.(\bigsqcup V)=
\csm{U\cup\{c\}}{\{1\}}\ominus\csm{U\cup\{c\}}{V}
\end{align*}

Note that, if $S$ contains some non-idempotent $c$, 
then $\csm{~.~}{~.~}$ is not strong. To see that (e*) is not satisfied,
put $U=V=\{c\}$ and compute
\begin{align*}
\csm{U\cup\{c\}}{\{1\}}\ominus\csm{U\cup\{c\}}{V}=c\ominus c.c\neq 0\\
\csm{U}{V\cup\{c\}}\ominus\csm{U}{V}=c.c\ominus c.c=0
\end{align*}
\end{example}

\section{Observables from compatibility support mappings}

The aim of this section is to prove that for every $S$ such that
$S\cup\{1\}$ admits a compatibility support mapping, then $S$ is coexistent.

The direct limit method used here is a dual of the projective limit method
introduced in \cite{CatDAlGiuPul:EAaPM}. See also \cite{Pul:AnoooM} for another
application of the projective limit method.

Several proofs in this section (Lemma 3 through Theorem 1) are very similar, 
or even the same, as in \cite{Jen:CiIEA}. The reason for this is that they
are basically an application of Lemma \ref{lemma:first}, which is
the Proposition 4 of \cite{Jen:CiIEA}. 
However, the author decided to include them here, to keep
the present paper more streamlined.

\begin{assumption}
In this section, we assume the following.
\begin{itemize}
\item $E$ is an effect algebra.
\item $S$ is a subset of $E$ with $1\in S$.
\item $\csm{~.~}{~.~}:\Fin(S)\times \Fin(S)\to S$ is a compatibility support mapping.
\end{itemize}
\end{assumption}

\begin{lemma}
\label{lemma:c1isc}
For all $c\in S$, $\csm{\{c\}}{\{1\}}=c$.
\end{lemma}
\begin{proof}
Put $U=V=\emptyset$ in condition (e) of Definition \ref{def:csm}. We see that
$$
\csm{\{c\}}{\{1\}}\ominus\csm{\{c\}}{\emptyset}=
	\csm{\emptyset}{\{c\}}\ominus\csm{\emptyset}{\emptyset}.
$$
By conditions (c) and (d), this implies that $\csm{\{c\}}{\{1\}}=c$.
\end{proof}
Let us write, for $A,X\in\Fin(S)$ such that $X\subseteq A$,
$$
D(X,A)=\csm{X}{\{1\}}\ominus\csm{X}{A\setminus X}.
$$
\begin{lemma}
\label{lemma:first}
Let $A,X\in\Fin(S)$, $X\subseteq A$ and
let $c\in S$ be such that
$c\not\in A$. Then
$$
D(X,A)=D(X,A\cup\{c\})\oplus D(X\cup\{c\},A\cup\{c\}).
$$
\end{lemma}
\begin{proof}
We see that
\begin{align*}
D(X,A\cup\{c\})=&\csm{X}{\{1\}}\ominus\csm{X}{\{c\}\cup(A\setminus X)}\\
D(X\cup\{c\},A\cup\{c\})=&\csm{X\cup\{c\}}{\{1\}}\ominus\csm{X\cup\{c\}}{A\setminus X}
\end{align*}
and, by condition (e) of Definition \ref{def:csm}, we see that
$$
\csm{X\cup\{c\}}{\{1\}}\ominus\csm{X\cup\{c\}}{A\setminus X}=
\csm{X}{\{c\}\cup(A\setminus X)}\ominus\csm{X}{A\setminus X}.
$$
Therefore,
\begin{align*}
D(X,A\cup\{c\})\oplus D(X\cup\{c\},A\cup\{c\})=\\
=(\csm{X}{\{1\}}\ominus\csm{X}{\{c\}\cup(A\setminus X)})
\oplus
(\csm{X}{\{c\}\cup(A\setminus X)}\ominus\csm{X}{A\setminus X})=\\
=\csm{X}{\{1\}}\ominus\csm{X}{A\setminus X}=D(X,A).
\end{align*}
\end{proof}
\begin{lemma}
\label{lemma:second}
Let $C,A,X\in\Fin(S)$ be such that $X\subseteq A$ and 
$C\cap A=\emptyset$.
Then $(D(X\cup Y,A\cup C))_{Y\subseteq C}$ is an orthogonal
family and
$$
\bigoplus_{Y\subseteq C}D(X\cup Y,A\cup C)=D(X,A).
$$
\end{lemma}
\begin{proof}
The proof goes by induction with respect to $|C|$.

For $C=\emptyset$, Lemma \ref{lemma:second} is trivially true.

Assume that Lemma \ref{lemma:second} holds for all $C$ 
with $|C|=n$ and let $c\in S,c\not\in A\cup C$.
Let us consider the family
$$
(D(X\cup Z,A\cup C\cup\{c\}))_{Z\subseteq C\cup\{c\}}.
$$
For every $Z\subseteq C\cup\{c\}$, either $c\in Z$ or
$c\not\in Z$, so either $Z=Y\cup\{c\}$ or $Z=Y$, for some
$Y\subseteq C$. Therefore, we can write
\begin{align*}
(D(X\cup Z,A\cup C\cup\{c\}))_{Z\subseteq C\cup\{c\}}=\\
(D(X\cup Y,A\cup C\cup\{c\}),
	D(X\cup Y\cup\{c\},A\cup C\cup\{c\}))_{Y\subseteq C}.
\end{align*}
By Lemma \ref{lemma:first},
$$
D(X\cup Y,A\cup C\cup\{c\})\oplus
	D(X\cup Y\cup\{c\},A\cup C\cup\{c\})=
D(X\cup Y,A\cup C).
$$
It only 
remains to apply the induction hypothesis to finish the proof.
\end{proof}
\begin{corollary}
\label{coro:decomposition}
For every $A\in\Fin(S)$, $(D(X,A))_{X\subseteq A}$ is a
decomposition of unit.
\end{corollary}
\begin{proof}
Obviously,
$$D(\emptyset,\emptyset)=\csm{\emptyset}{\{1\}}\ominus
	\csm{\emptyset}{\emptyset}=1\ominus 0=1.
$$

By Lemma \ref{lemma:second},
$$
\bigoplus_{X\subseteq A}(D(\emptyset\cup X,\emptyset\cup A))
=D(\emptyset,\emptyset).
$$
\end{proof}
\begin{corollary}
\label{coro:alphaAobservable}
For every $A\in\Fin(S)$, the mapping
$\alpha_A:2^{(2^A)}\to E$ given by
$$
\alpha_A(\mathbb X)=\bigoplus_{X\in\mathbb X}D(X,A)
$$
is a simple observable.
\end{corollary}
\begin{proof}
The atoms of $2^{(2^A)}$ are of the form $\{X\}$, where
$X\subseteq A$.
By Corollary \ref{coro:decomposition},
$(\alpha_A(\{X\}):X\subseteq A)$ is a decomposition of unit;
the remainder of the proof is trivial.
\end{proof}

For $A,B\in\Fin(S)$ with $A\subseteq B$,
let us define mappings $g^A_B:2^{(2^A)}\to 2^{(2^B)}$
$$
g^A_B(\mathbb X)=\{X\cup C_0:X\in\mathbb X
	\text{ and }C_0\subseteq (B\setminus A)\}
$$
and let us write $\mathcal G$ for the collection of all
such mappings.

It is an easy exercise to prove that every $g^A_B\in\mathcal G$ is an injective 
homomorphism of Boolean algebras and that
$((2^{(2^A)}:A\in\Fin(S)),\mathcal G)$ is a direct family of Boolean algebras.

Let us prove that the mappings $g^A_B$ behave well with respect to the
observables $\alpha_A$ and $\alpha_B$.

\begin{lemma}
\label{lemma:d1commutes}
Let $A,B\in\Fin(S)$ with $A\subseteq B$. The diagram
\begin{center}
\includegraphics{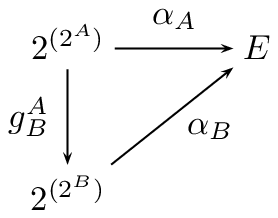}
\end{center}
commutes.
\end{lemma}
\begin{proof}
For all $\mathbb X\in 2^{(2^A)}$,
\begin{align*}
\alpha_B(g^A_B(\mathbb X))=
	\alpha_B(\{X\cup C_0:X\in\mathbb X
	\text{ and }C_0\subseteq (B\setminus A)\})=\\
=\bigoplus(
	D(X\cup C_0,B):X\in\mathbb X
	\text{ and }C_0\subseteq (B\setminus A)
	)=\\
=\bigoplus_{X\in\mathbb X}\Bigl(
\bigoplus_{C_0\subseteq (B\setminus A)}
D(X\cup C_0,B)\Bigr)
\end{align*}
Put $Y:=C_0$, $C:=B\setminus A$; by Lemma \ref{lemma:second},
$$
\bigoplus_{C_0\subseteq (B\setminus A)}
D(X\cup C_0,B)=D(X,A).
$$
Therefore,
$$
\alpha_B(g^A_B(\mathbb X))=\bigoplus_{X\in\mathbb X} D(X,A)=
\alpha_A(\mathbb X)
$$
and the diagram commutes.
\end{proof}

\begin{corollary}
\label{coro:simplerange}
For every $B\in\Fin(S)$, $B$ is a subset of the range of
$\alpha_B$.
\end{corollary}
\begin{proof}
We need to prove that every $a\in B$ is an element of the range
of $\alpha_B$. For $B=\emptyset$, this is trivial.

Suppose that
$B$ is nonempty and let $a\in B$. Let $A=\{a\}$.
and let $X=g^A_B(\{\{a\}\})$.
By
Lemma \ref{lemma:d1commutes}, 
$$
\alpha_B(X)=\alpha_B(g^A_B(\{\{a\}\}))=\alpha_A(\{\{a\}\}),
$$
and we see that, by (c) of Definition \ref{def:csm} and by Lemma 
\ref{lemma:c1isc}
$$
\alpha_A(\{\{a\}\})=\alpha_{\{a\}}(\{\{a\}\})=D(\{a\},\{a\})=
\D{\{a\}}{\{a\}}=a\ominus 0=a.
$$
\end{proof}

\begin{theorem}
\label{thm:obsfromcsm}
Let $E$ be an effect algebra, let $S\subseteq E$.
If $S\cup\{1\}$ admits a compatibility
support mapping, then $S$ is coexistent.
\end{theorem}
\begin{proof}
Suppose that $S\cup\{1\}$ admits a compatibility support mapping.
Let us construct
$F_B(S)$ as the direct limit of the direct family $(2^{2^A}:A\in\Fin(S))$,
equipped with morphisms of the type $g^A_B$.
After that, we shall define an observable $\alpha:F_B(S)\to E$.

Consider the set
$$
\Gamma_S=\bigcup_{A\in\Fin(S)}\{(\mathbb X,A):\mathbb X\subseteq 2^A\}
$$
and define on it a binary relation $\equiv$ by
$(\mathbb X,A)\equiv(\mathbb Y,B)$ if and only if
$g^A_{A\cup B}(\mathbb X)=g^B_{A\cup B}(\mathbb Y)$, that means
$$
\{X\cup C_A:X\in\mathbb X\and C_A\subseteq A\cup B\setminus A\}=
\{Y\cup C_B:Y\in\mathbb Y\and C_B\subseteq A\cup B\setminus B\}.
$$
Then $F_B(S)=\Gamma_S/\equiv$ and the operations
on $F_B(S)$ are defined by
$$
[(\mathbb X,A)]_\equiv\vee[(\mathbb Y,B)]_\equiv=
[(g^A_{A\cup B}(\mathbb X)\cup g^B_{A\cup B}(\mathbb Y),A\cup B)]_\equiv
$$
and similarly for the other operations. Then
$F_B(S)$ is a direct limit of Boolean algebras, hence a Boolean algebra.

Let $\alpha_S:F_B(S)\to E$ be a mapping given by the rule 
$\alpha_S([(\mathbb X,A)]_\equiv)=\alpha_A(\mathbb X)$.
We shall prove that $\alpha_S$ is an observable.

Let us prove $\alpha_S$ is well-defined.
Suppose that $(\mathbb X,A)\equiv (\mathbb Y,B)$, that means,
$g^A_{A\cup B}(\mathbb X)=g^B_{A\cup B}(\mathbb Y)$.
By Lemma \ref{lemma:d1commutes},
$$
\alpha_A(\mathbb X)=\alpha_{A\cup B}(g^A_{A \cup B}(\mathbb X))
$$
and
$$
\alpha_B(\mathbb Y)=\alpha_{A\cup B}(g^B_{A \cup B}(\mathbb Y)),
$$
hence $\alpha_S$ is a well-defined mapping.

Let us prove that $\alpha_S$ is an observable.
The bounds of the Boolean algebra 
$F_B(S)$ are $[(\emptyset,A)]_\equiv$ and
$[(2^A,A)]_\equiv$, where $A\in Fin(S)$. Obviously,
by Corollary \ref{coro:alphaAobservable},
$$
\alpha_S([(\emptyset,A)]_\equiv)=\alpha_A(\emptyset)=0
$$
and
$$
\alpha_S([(2^A,A)]_\equiv)=\alpha_A(2^A)=1.
$$
Let $[(\mathbb X,A)]_\equiv$ and $[(\mathbb Y,B)_\equiv]$ be disjoint elements 
of $F_B(S)$, that is, 
$g^A_{A\cup B}(\mathbb X)\cap g^B_{A\cup B}(\mathbb Y)=\emptyset$.
Then
\begin{align*}
\alpha_S([(\mathbb X,A)]_\equiv\vee [(\mathbb Y,B)]_\equiv)=
\alpha_S([g^A_{A\cup B}(\mathbb X)
	\cup g^B_{A\cup B}(\mathbb Y),A\cup B]_\equiv)=\\
=\alpha_{A\cup B}
	(g^A_{A\cup B}(\mathbb X)\cup g^B_{A\cup B}(\mathbb Y)).
\end{align*}
Since $\alpha_{A\cup B}$ is an observable,
$$
\alpha_{A\cup B}
	(g^A_{A\cup B}(\mathbb X)\cup g^B_{A\cup B}(\mathbb Y))=
\alpha_{A\cup B}
	(g^A_{A\cup B}(\mathbb X))\oplus 
	\alpha_{A\cup B}(g^B_{A\cup B}(\mathbb Y)).
$$
It remains to observe that
$$
\alpha_{A\cup B}
	(g^A_{A\cup B}(\mathbb X))=
\alpha_S([(\mathbb X,A)]_\equiv)
$$
and that
$$
\alpha_{A\cup B}
	(g^B_{A\cup B}(\mathbb Y))=
\alpha_S([(\mathbb Y,B)]_\equiv).
$$

Let us prove that the range of $\alpha_S$ includes $S$. Let $a\in S$.
By Corollary \ref{coro:simplerange}, the range of $\alpha_{\{a\}}$ includes
$a$ and, by an obvious direct limit argument, the range of $\alpha_{\{a\}}$ is a subset of the range
of $\alpha_S$. 
\end{proof}

\section{Compatibility support mappings from observables}

The aim of the single theorem of this section is to prove that
every subset $S$ of the range of an observable admits a strong
compatibility support mapping.

\begin{theorem}\label{thm:csmfromobs}
For every coexistent subset $S$ of an effect algebra
$E$, $S\cup\{1\}$ admits a strong compatibility
support mapping. 
\end{theorem}
\begin{proof}
Let $B$ be a Boolean algebra and let 
$\alpha:B\to E$ be an observable, let $S$ be a subset of the
range of $\alpha$.

For every $a\in S\cup\{1\}$, fix an element
$p_a\in\alpha^{-1}(a)$ and define
$$
\csm{U}{V}=\alpha((\bigwedge_{a\in U} p_a)\wedge(\bigvee_{b\in V}p_b)).
$$

Let us check the condition in the definition of a strong compatibility
support mapping.
Let $c\not\in U,V$. Then
\begin{align*}
\csm{U\cup\{c\}}{\{1\}}\ominus\csm{U\cup\{c\}}{V}=\\
=\alpha((\bigwedge_{a\in U} p_a)\wedge p_c)\ominus
\alpha(((\bigwedge_{a\in U} p_a)\wedge p_c)\wedge(\bigvee_{b\in V}p_b)).
\end{align*}
To simplify the matters, write
\begin{align*}
m_U=(\bigwedge_{a\in U} p_a)\\
j_V=(\bigvee_{b\in V}p_b)
\end{align*}
We can write
\begin{align*}
\alpha((\bigwedge_{a\in U} p_a)\wedge p_c)\ominus
\alpha(((\bigwedge_{a\in U} p_a)\wedge p_c)\wedge(\bigvee_{b\in V}p_b))=
\alpha(m_U\wedge p_c)\ominus\alpha(m_U\wedge p_c\wedge j_V)=\\
=\alpha((m_U\wedge p_c)\ominus(m_U\wedge p_c\wedge j_V))\\
\end{align*}
Similarly,
$$
\csm{U}{V\cup\{c\}}\ominus\csm{U}{V}=
\alpha((m_U\wedge(p_c\vee j_V))\ominus(m_U\wedge j_V)).
$$
Since $B$ is a Boolean algebra,
\begin{align*}
(m_U\wedge p_c)\ominus(m_U\wedge p_c\wedge j_V)=
(m_U\wedge(p_c\vee j_V))\ominus(m_U\wedge j_V)
\end{align*}

The remaining conditions are trivial to check.
\end{proof}

Let us note that, if we start with a non-strong compatibility support mapping,
apply Theorem \ref{thm:obsfromcsm} to construct an observable and then
apply Theorem \ref{thm:csmfromobs} to construct a compatibility support mapping,
we cannot obtain the compatibility support mapping we started with, since
Theorem \ref{thm:csmfromobs} always produces a strong compatibility support mapping.

\section{Properties of strong compatibility support mappings}

The aim of this section is to prove that several properties of the Example
\ref{ex:joinmeet} are valid for all strong compatibility support mappings. 
It
remains open whether and which of these properties are valid for all
compatibility support mappings.

The main vehicle here is Proposition \ref{prop:csmfromD}, that is interesting
in its own right: it shows that, for a given $S$, every strong compatibility
support mapping on $S$ is determined by its $D(~.~,~.~)$.

\begin{assumption}
In this section, we assume the following.
\begin{itemize}
\item $E$ is an effect algebra.
\item $S$ is a subset of $E$ with $1\in S$.
\item $\csm{~.~}{~.~}:\Fin(S)\times \Fin(S)\to S$ is a strong 
compatibility support mapping.
\end{itemize}
\end{assumption}

\begin{lemma}
\label{lemma:nondisjoint}
If $U,V$ are not disjoint, then $\csm{U}{V}=\csm{U}{\{1\}}$.
\end{lemma}
\begin{proof}
Let $c\in U\cap V$. This implies that $U\cup\{c\}=U$ and
$V\cup\{c\}=V$. Therefore, by (e*),
$$
\csm{U}{\{1\}}\ominus\csm{U}{V}=
	\csm{U}{V}\ominus\csm{U}{V}=0,
$$
hence $\csm{U}{V}=\csm{U}{\{1\}}$.
\end{proof}
\begin{lemma}
\label{prop:trop}
$\csm{U\cup\{c\}}{\{1\}}=\csm{U}{\{c\}}$.
\end{lemma}
\begin{proof}
Put $V=\emptyset$ in (e*):
$$
\csm{U\cup\{c\}}{\{1\}}\ominus\csm{U\cup\{c\}}{\emptyset}=
	\csm{U}{\{c\}}\ominus\csm{U}{\emptyset}.
$$
By condition (c), 
$\csm{U\cup\{c\}}{\emptyset}=\csm{U}{\emptyset}=0$, therefore
$$
\csm{U\cup\{c\}}{\{1\}}=\csm{U}{\{c\}}.
$$
\end{proof}

\begin{proposition}
\label{prop:csmfromD}
Let $U,V\subseteq S$.
\begin{enumerate}
\item If $U\cap V\not =\emptyset$,
then $\csm{U}{V}=\csm{U}{\{1\}}=D(U,U)$.
\item If
$U\cap V=\emptyset$,
then
$$
\csm{U}{V}=\bigoplus_{\emptyset\neq Y\subseteq V}
	D(U\cup Y,U\cup V).
$$
\end{enumerate}
\end{proposition}
\begin{proof}~

(1)
By Proposition \ref{lemma:nondisjoint},
$\csm{U}{V}=\csm{U}{\{1\}}$ and
$$
D(U,U)=\csm{U}{\{1\}}\ominus\csm{U}{\emptyset}=
	\csm{U}{\{1\}}\ominus 0=\csm{U}{\{1\}}.
$$

(2) By Lemma \ref{lemma:second},
$$
D(U,U)=\bigoplus_{Y\subseteq V}
	D(U\cup Y,U\cup V).
$$
Therefore,
$$
D(U,U)\ominus D(U,U\cup V)=
	\bigoplus_{\emptyset\neq Y\subseteq V}
		D(U\cup Y,U\cup V).
$$
Moreover,
\begin{align*}
D(U,U)\ominus D(U,U\cup V)=
	(\csm{U}{\{1\}}\ominus\csm{U}{\emptyset})\ominus
	(\csm{U}{\{1\}}\ominus\csm{U}{V})=\\
	=\csm{U}{V}\ominus\csm{U}{\emptyset}=\csm{U}{V}\ominus 0=
	\csm{U}{V}.
\end{align*}
\end{proof}

\begin{proposition}
\label{prop:formera}
If $U_1\subseteq U_2$, then $\csm{U_1}{V}\geq\csm{U_2}{V}$.
\end{proposition}
\begin{proof}~

(Case 1)
Suppose that $U_1\cap V\neq\emptyset$. Then $U_2\cap V\neq\emptyset$.
By Proposition \ref{prop:csmfromD} and Lemma \ref{lemma:second},
$$
\csm{U_1}{V}=D(U_1,U_1)=
	\bigoplus_{Y\subseteq U_2\setminus U_1}D(U_1\cup Y,U_2)\geq
	D(U_2,U_2)=\csm{U_2}{V}.
$$

(Case 2)
Suppose that $U_2\cap V=\emptyset$. Then $U_1\cap V=\emptyset$.
By Proposition
\ref{prop:csmfromD},
$$
\csm{U_1}{V}=\bigoplus_{\emptyset\neq Y\subseteq V}
	D(U_1\cup Y,U_1\cup V).
$$
By Lemma \ref{lemma:second}, 
for every $\emptyset\neq Y\subseteq V$,
$$
D(U_1\cup Y,U_1\cup V)=\bigoplus_{W\subseteq U_2\setminus U_1}
D(U_1\cup Y\cup W,U_2\cup V).
$$
Obviously (put $W=U_2\setminus U_1$), this implies
that
$$
D(U_1\cup Y,U_1\cup V)\geq D(U_2\cup Y,U_2\cup V),
$$
hence we may write
$$
\csm{U_1}{V}=\bigoplus_{\emptyset\neq Y\subseteq V}
	D(U_1\cup Y,U_1\cup V)\geq
\bigoplus_{\emptyset\neq Y\subseteq V}
	D(U_2\cup Y,U_2\cup V).
$$
It remains to apply Proposition \ref{prop:csmfromD} again:
$$
\bigoplus_{\emptyset\neq Y\subseteq V}
	D(U_2\cup Y,U_2\cup V)=\csm{U_2}{V}.
$$

(Case 3) Suppose that $U_1\cap V=\emptyset$ and
$U_2\cap V\neq\emptyset$. By Proposition \ref{prop:csmfromD},
$$
\csm{U_1}{V}=\bigoplus_{\emptyset\neq Y\subseteq V}D(U_1\cup Y,U_1\cup V).
$$
By Lemma \ref{lemma:second},
$$
D(U_1\cup Y,U_1\cup V)=
	\bigoplus_{W\subseteq U_2\setminus(U_1\cup V)}
	D(U_1\cup W\cup Y,U_2\cup V)
$$
We can put $W=U_2\setminus(U_1\cup V)$, proving that
$$
D(U_1\cup Y,U_1\cup V)\geq D( (U_2\setminus V)\cup Y,U_2\cup V).
$$
Therefore,
\begin{align*}
\csm{U_1}{V}=
	\bigoplus_{\emptyset\neq Y\subseteq V}
	D(U_1\cup Y,U_1\cup V)\geq
	\bigoplus_{\emptyset\neq Y\subseteq V}
	D( (U_2\setminus V)\cup Y,U_2\cup V)\geq\\
	\bigoplus_{V\cap U_2\subseteq Y\subseteq V}
	D( (U_2\setminus V)\cup Y,U_2\cup V).
\end{align*}
For every $V\cap U_2\subseteq Y\subseteq V$, there is exactly one
$Z\subseteq V\setminus U_2$ such that
$$
(U_2\setminus V)\cup Y=U_2\cup Z.
$$
Thus, we can rewrite
$$
\bigoplus_{V\cap U_2\subseteq Y\subseteq V}
	D( (U_2\setminus V)\cup Y,U_2\cup V)=
\bigoplus_{Z\subseteq V\setminus U_2}
	D(U_2\cup Z,U_2\cup V).
$$
By Lemma \ref{lemma:second} and Proposition \ref{prop:csmfromD},
$$
\bigoplus_{Z\subseteq V\setminus U_2}
	D(U_2\cup Z,U_2\cup V)=D(U_2,U_2)=\csm{U_2}{V}.
$$
\end{proof}

\begin{proposition}
\label{prop:lowbound}
$\csm{U}{\{1\}}$ is a lower bound of $U$.
\end{proposition}
\begin{proof}
Any element is a lower bound of $\emptyset$.

Suppose that the proposition is true for some $U$
and pick $c\in S\setminus U$.
By Proposition \ref{prop:formera},
$$
\csm{U\cup\{c\}}{\{1\}}\leq\csm{U}{\{1\}}.
$$
By the induction hypothesis, $\csm{U}{\{1\}}$
is a lower bound of $U$. It remains to prove that
$\csm{U\cup\{c\}}{\{1\}}\leq c$. By Proposition
\ref{prop:trop},
$\csm{U\cup\{c\}}{\{1\}}=\csm{U}{\{c\}}$. By Proposition
\ref{prop:formera} 
and condition (d), 
$\csm{U}{\{c\}}\leq\csm{\emptyset}{\{c\}}=c$.
\end{proof}
\begin{corollary}
$\csm{U}{V}$ is a lower bound of $U$.
\end{corollary}
\begin{proof}
By Proposition \ref{prop:lowbound},
$\csm{U}{\{1\}}$ is a lower bound of $U$. By condition (b),
$\csm{U}{V}\leq\csm{U}{\{1\}}$.
\end{proof}
\begin{proposition}
$\csm{\emptyset}{V}$ is an upper bound of $V$.
\end{proposition}
\begin{proof}
Any element is an upper bound of $\emptyset$.

Suppose that the proposition is true for some $V$
and pick $c\in S\setminus V$.
By condition (a), 
$$
\csm{\emptyset}{V}\leq\csm{\emptyset}{V\cup\{c\}}
$$
and by induction hypothesis, $\csm{\emptyset}{V}$ is an
upper bound of $V$. It remains to prove that
$c\leq \csm{\emptyset}{V\cup\{c\}}$.

Put $U=\emptyset$ in condition (e*):
$$
\csm{\{c\}}{\{1\}}\ominus\csm{\{c\}}{V}=
	\csm{\emptyset}{V\cup\{c\}}\ominus\csm{\emptyset}{V}.
$$
Add $\csm{\emptyset}{V}$ to both sides to obtain
$$
(\csm{\{c\}}{\{1\}}\ominus\csm{\{c\}}{V})\oplus\csm{\emptyset}{V}=
	\csm{\emptyset}{V\cup\{c\}}.
$$
As $\csm{\{c\}}{V}\leq\csm{\emptyset}{V}$,
$$
\csm{\{c\}}{\{1\}}\leq
	(\csm{\{c\}}{\{1\}}\ominus\csm{\{c\}}{V})\oplus\csm{\emptyset}{V}.
$$
By Lemma \ref{lemma:c1isc}, 
$\csm{\{c\}}{\{1\}}=c$.
\end{proof}

\section{Compatibility support mappings and witness mappings}

Let $(G,\leq)$ be a partially
ordered abelian group and 
$u\in G$ be a positive element.
For $0\leq a,b\leq u$, define $a\oplus b$ if and only if
$a+b\leq u$ and put $a\oplus b=a+b$.  With such a partial operation $\oplus$, the
closed interval 
$$
[0,u]_G=\{x\in G:0\leq x\leq u\}
$$ 
becomes an effect algebra $([0,u]_G,\oplus,0,u)$.  Effect
algebras which arise from partially ordered abelian groups in this way are
called {\em interval effect algebras}, see \cite{BenFou:IaSEA}.

Let $E$ be an interval effect algebra in a partially ordered abelian group $G$.
Let $S\subseteq E$. Let us write $\Fin(S)$ for the set of all
finite subsets of $S$. We write $I(\Fin(S))$ for the set of all comparable
elements of the poset $(Fin(S),\subseteq)$, that means,
$$
I(\Fin(S))=\{(X,Y)\in\Fin(S)\times\Fin(S):X\subseteq Y\}.
$$

For every mapping $\beta:\Fin(S)\to G$, we define a
mapping $D_\beta:I(\Fin(S))\to G$.
For $(X,A)\in I(\Fin(S))$,
the value
$D_\beta(X,A)\in G$ is given by the rule
$$
D_\beta(X,A):=\sum_{X\subseteq Z\subseteq A}(-1)^{|X|+|Z|}\beta(Z).
$$

In \cite{Jen:CiIEA}, we introduced and studied the following notion:
\begin{definition}\label{def:cm}
Let $E$ be an interval effect algebra.
We say that a mapping $\beta:\Fin(S)\to E$ is a {\em witness
mapping for $S$} if and only if the following conditions are satisfied.
\begin{enumerate}
\item[(A1)]$\beta(\emptyset)=1$,
\item[(A2)]for all $c\in S$, $\beta(\{c\})=c$,
\item[(A3)]for all $(X,A)\in I(\Fin(S))$, $D_\beta(X,A)\geq 0$.
\end{enumerate}
\end{definition}
We proved there, that a subset $S$ of an interval effect algebra $E$
is coexistent if and only if there is a witness 06 $\beta:\Fin(S)\to E$.

The aim of this section is to explore the connection between the
notion of a witness mapping and the notion of compatibility support mappings.

\begin{proposition}
\label{prop:wmfromcsm}
Let $E$ be an interval effect algebra, let $S$ be a subset of $E$ with $1\in S$.
Suppose there is a compatibility support mapping 
$\csm{~.~}{~.~}:\Fin(S)\times \Fin(S)\to S$.
Then $\beta:\Fin(S)\to E$, given by $\beta(X)=\csm{X}{\{1\}}$ is
a witness mapping and $D(X,A)=D_\beta(X,A)$, for all $(X,A)\in I(\Fin(S))$.
\end{proposition}
\begin{proof}
We see that, by the condition (d) of Definition \ref{def:csm},
$$
\beta(\emptyset)=\csm{\emptyset}{\{1\}}=1,
$$
so the condition 
(A1) of Definition \ref{def:cm} is satisfied.
By Lemma \ref{lemma:c1isc}, 
$$
\beta(\{c\})=\csm{\{c\}}{\{1\}}=c,
$$
hence (A2) is satisfied.

For the proof of (A3), it suffices to prove that
$D(X,A)=D_\beta(X,A)$, for all $(X,A)\in I(\Fin(S))$.
The positivity of $D_\beta$ then follows from the positivity
of $D$. The proof goes by induction with respect to $|A\setminus X|$.

If $|A\setminus X|=0$, then $A=X$ and
$$
D_\beta(X,A)=\beta(X)=\csm{X}{\{1\}}=
\csm{X}{\{1\}}\ominus 0=
\csm{X}{\{1\}}\ominus \csm{X}{\emptyset}=D(X,A).
$$

Suppose that $D(X,A)=D_\beta(X,A)$, for all $(X,A)\in I(\Fin(S))$
such that $|A\setminus X|=n$.
Let $(Y,B)\in I(\Fin(S))$ be such that $|B\setminus Y|=n+1$.
Pick $c\in B\setminus Y$ and put $X=Y$, $A=B\setminus\{c\}$.

By Lemma 1 of \cite{Jen:CiIEA}, for any mapping $\beta:\Fin(S)\to E$,
for all $(X,A)\in I(\Fin(S))$ and for all $c\in S\setminus A$, the
following equality is satisfied:
$$
D_\beta(X,A)=D_\beta(X,A\cup\{c\})+D_\beta(X\cup\{c\},A\cup\{c\}).
$$
Therefore,
$$
D_\beta(Y,B)=D_\beta(X,A\cup\{c\})=
D_\beta(X,A)\ominus D_\beta(X\cup\{c\},A\cup\{c\}).
$$
By the induction hypothesis, $D_\beta(X,A)=D(X,A)$ and
$D_\beta(X\cup\{c\},A\cup\{c\})=D_\beta(X\cup\{c\},A\cup\{c\})$.
Thus,
$$
D_\beta(Y,B)=
D(X,A)\ominus D(X\cup\{c\},A\cup\{c\}).
$$
By Lemma \ref{lemma:first}, 
$$
D(X,A)\ominus D(X\cup\{c\},A\cup\{c\})=D(X,A\cup\{c\})=D_\beta(Y,B).
$$
\end{proof}

The following problem remains open.

\begin{problem}
Let  $E$ be an effect algebra, let $S\subseteq E$, let $\beta:\Fin(S)\to E$ be
a witness mapping. Is there always a compatibility support mapping
$\csm{~.~}{~.~}:\Fin(S)\times \Fin(S)\to S$ such that $\beta(X)=\csm{X}{\{1\}}$?
\end{problem}


\end{document}